\documentclass[12pt, reqno]{amsart}
\usepackage{amsmath, amsthm, amscd, amsfonts, amssymb, graphicx, color}

\newtheorem{df}{Definition}[section]
\newtheorem{thm}[df]{Theorem}
\newtheorem{pro}[df]{Proposition}

\begin{document}
\setcounter{page}{1}

\title[Joint Spectral Radius]{On The joint spectral radius of\\ a nilpotent Lie algebra of matrices}

\author{Enrico Boasso}

\begin{abstract}For a complex nilpotent finite dimensional Lie algebra of matrices,
and a Jordan-H\"older basis of it, we prove a spectral radius formula which
extends  a well-known result for commuting matrices. \end{abstract}
\maketitle

\section{Introduction}

\indent Let $T=(T_1,\ldots ,T_n)$ be an $n$-tuple of $d\times d$-complex matrices. 
A point $\lambda\in\Bbb C^n$ is in the joint point spectrum of $T$, $\sigma_{pt}(T)$,
if there exists a nonzero  element $x\in\Bbb C^d$ with the property $T_i(x)=\lambda_ix$, 
$1\le i\le n$. Given $p$ such that $1\le p\le \infty$, R. Bhatia and T. Bhattacharyya introduced in [1] the algebraic spectral radius 
of an $n$-tuple $T$, $\varrho_p(T)$, whose definition depends of the usual $p$-norm
of $\Bbb C^d$, and proved that if $T$ is an $n$-tuple of commuting matrices,
then the algebraic spectral radius coincides with the geometrical spectral radius, i.e.,
$\varrho_p(T)=r_p(T)=\hbox{max }\{\mid \lambda\mid_p\colon \lambda\in \sigma_{pt}(T)\}$ 
(see [1] or Section 2 for more details). This is a generalization 
of the well-known spectral radius formula for a single matrix;
for $p=2$, it was proved by M. Ch\={o} and T. Huruya in [6].\par

\indent M. Ch\={o} and M. Takaguchi 
proved in [7] that if $T$ is a commuting $n$-tuple of matrices, then $\sigma_{pt}(T)=Sp(T,{\Bbb C}^d)$, where $Sp(T,{\Bbb C}^d)$ denotes the Taylor joint
spectrum of $T$ (see [12]). A. McIntosh, A. Pryde and W. Ricker, as a consequence
of a more general result which also concerns infinite dimensional spaces, extended in [9]
the above identity to many other joint spectra including  the commutant, the bicommutant
and the Harte joint spectra.\par

\indent On the other hand, in [4] we defined a joint spectrum, $Sp(L,E)$, for complex solvable finite 
dimensional Lie algebras $L$ of operators acting on a Banach space $E$. We proved that $Sp(L,E)$ is a compact nonempty
subset of $L^*$ satisfying  the projection property for ideals. 
Moreover, when $L$ is a commutative algebra, $Sp(L,E)$ reduces to the Taylor 
joint spectrum in the following sense. If $\dim L = n$ and $\{ T_i\}_{(1\le i\le n)}$
is a basis of $L$, then 
$\{(f(T_1),\ldots ,f(T_n))\colon f\in Sp(L,E)\} = Sp(T,E)$ for $T=(T_1, \ldots , T_n)$ i.e., $Sp(L,E)$ in terms of the
basis of $L^*$ dual to $\{T_i\}_{(1\le i\le n)}$ coincides with the Taylor
joint spectrum of the n-tuple $T$. Furthermore, in [3] we also extended to complex
solvable finite dimensional Lie algebras the S\l odkowski joint spectra $\sigma_{\delta ,k}$ and $\sigma_{\pi ,k}$,
and we proved the most important properties of a spectral theory, i.e., the compactness, 
nonemptiness and the projection property for ideals.\par  

\indent For $E$ finite dimensional, in [3] we extended the characterization of [7], and 
partially that of [9], to complex nilpotent Lie algebras acting on $E$.
Indeed, we proved that $Sp(L,E)$, $\sigma_{\delta ,k}(L,E)$ and $\sigma_{\pi ,k}(L,E)$ coincide with the set of all weights of the Lie algebra 
$L$ for the vector space $E$, which is the intrinsic algebraic description of 
the joint point spectrum, and showed that in suitable bases of $L$ and $L^*$ it reduces
to the joint point spectrum of the basis of $L$ (see [3] or Section 2 for details). This also extended the characterization 
of [7] and [9] to the S\l odkowski joint spectra of an $n$-tuple of commuting matrices. \par

\indent Thus, if $L$ is a complex nilpotent finite dimensional Lie algebra acting on a 
complex finite dimensional vector space $E$, and if instead of considering the elements of $Sp(L,E)$ as linear 
functionals on $L$ we work with their coordinates in a basis of $L^*$,
dual to a suitable basis of $L$, then, as in the former basis $Sp(L,E)$ reduces to
the joint point spectrum of the latter basis, it is possible to compute the geometrical and the algebraic spectral radii
of $L$ with respect to its basis, and to look for a generalization of the main
result of [1]. In this article we extend the spectral radius formula of R. Bhatia and T. Bhattacharyya
for the case under consideration. The argument is quite elementary, and it furnishes another proof of the  formula for the commutative case.\par

\indent The paper is organized as follows. In Section 2 we review several definitions
and results which we need. In Section 3, we prove our main theorem and
study some examples in order to show that in the solvable non nilpotent case the spectral radius
formula fails.
                                            
\section{Preliminaries}
\indent We briefly recall several definitions and results related to the 
spectrum of a Lie algebra(see [4]). Although in [4] we considered complex 
solvable finite dimensional Lie algebras of linear bounded 
operators acting on a Banach space, for our purpose we restrict ourselves to 
the case of complex finite dimensional nilpotent Lie algebras of linear transformations 
defined on finite dimensional vector spaces. Moreover, as in this case the S\l odkowski joint spectra and the Taylor joint spectrum coincide,
we concentrate on the latter; for more information about
the S\l odkowski joint spectra see [11] and [2].\par

\indent From now on $E$ denotes a complex finite dimensional vector space,
${\mathcal L}(E)$ the algebra of all linear transformations defined on $E$, and 
$L$ a complex nilpotent finite dimensional Lie subalgebra of 
${{\mathcal L}(E)}^{op}$,  the algebra ${{\mathcal L}(E)}$ with opposite product. 
Such an algebra is called a nilpotent Lie algebra of 
linear transformations in $E$. If $ \dim L= n$ and $f$ is a 
character of $L$, i.e., $f\in L^*$ and $f(L^2) = 0$, where $L^2 = \{ [x,y]\colon x,
 y \in L\}$, consider the   chain complex
$(E\otimes\wedge L, d(f))$, where $\wedge L$ denotes
the exterior algebra of $L$ and
$$
d_p (f)\colon E\otimes\wedge^p L\rightarrow E\otimes\wedge^{p-1} L,   
$$

$$
d_p (f) e\langle x_1\wedge\dots\wedge x_p\rangle  = \sum_{k=1}^p (-1)^{k+1}e(x_k-f(x_k))\langle x_1\wedge\ldots\wedge\hat{x_k}\wedge\dots\wedge x_p\rangle
$$
                                                        $$ + \sum_{1\le k< l\le p} (-1)^{k+l}e\langle [x_k, x_l]\wedge x_1\wedge\ldots\wedge\hat{x_k}\wedge\ldots\wedge\hat{x_l}\wedge\ldots\wedge  x_p\rangle ,$$
where $\hat{ }$ means deletion. If $p\le 0$ or $p\ge {n+1}$, 
we  define $d_p (f) =0$.\par
 
\indent  If we denote by $H_*((E\otimes\wedge L,d(f)))$ the homology of the complex
$(E\otimes\wedge L, d(f))$, we may state our first definition.\par

\begin{df} With $E$, $L$ and $f$ as above, the
set $\{f\in L^*\colon f(L^2) =0, H_*((E\otimes\wedge L,d(f)))\neq 0\}$ is the joint spectrum
of $L$ acting on $E$, and it is denoted by $Sp(L,E)$.\end{df}

\indent As already mentioned, in [4] we proved that $Sp(L,E)$ is a compact nonempty
subset of $L^*$, which reduces, in the sense explained in the Introduction, to the Taylor joint spectrum when $L$ is a 
commutative algebra. Moreover, if $I$ is an ideal of $L$, and $\pi$ denotes 
the projection map from $L^*$ to $I^*$, then,
$$
Sp(I,E) = \pi (Sp(L,E)),
$$

\noindent i.e., the projection property for ideals still holds. In this connection, we wish to mention the paper of C. Ott ([10]) who pointed out
a gap in [4] in the proof of this result and gave another proof.\par

\indent We recall the most important results of the theory of weight spaces, essentially Theorem 7 and 12, Chapter II of [8]. For a complete exposition see [8, Chapter II].\par

\indent Let $L$ and $E$ be as above. A weight of $L$ for $E$ is a mapping,
$\alpha: L \to \mathbb{C}$ such that there exists a non zero
vector $v$ in $E$ with the following property: for each $x$ in $L$ there is $m_{v,x}$ in $\mathbb{N}$ such that $(x-\alpha (x))^{m_{v,x}}(v)= 0$. 
The set of vectors, zero included, which satisfy this condition is a subspace of $E$, denoted by $E_{\alpha}$ and called the
weight space of $E$ corresponding to the weight $\alpha$.

\indent Under our assumptions we have the following properties
(see [8, Chapter II, Theorems 7 and 12]):\par
\noindent (i) the weights are linear functions on $L$ which vanish on $L^2$, 
i.e.,  they are characters of $L$,\par
\noindent (ii) $E$ has only a finite number of distinct weights; the weight 
spaces are submodules, and $E$ is the direct sum of them,\par
\noindent (iii) for each weight $\alpha$, the restriction of any $x \in L$ to
$E_{\alpha}$ has only one characteristic root, $\alpha (x)$, with certain
multiplicity,\par
\noindent (iv) there is a basis of $E$ such that for each weight $\alpha$  and each $x\in L$ the
matrix of $x_{\alpha}$, the restriction of $x$ to $E_{\alpha}$, is
$$
x_{\alpha}=\begin{pmatrix}
                    \alpha (x)&*\\
                    0&\alpha (x)\\
                                       \end{pmatrix}.
$$         
\indent We now recall the following theorem of [3], which will
be crucial for our main result.\par

\noindent {\bf Theorem A} {\it Let $E$ be a complex finite dimensional vector space, and 
$L$ a complex fintite dimensional nilpotent Lie subalgebra of ${\mathcal L}(E)^{op}$. Then, }
$$
Sp(L,E)=\{\alpha\in L^*\colon \alpha\hbox{ is a weight of L for E}\}.
$$

\indent We observe that the right hand side set is a generalization
of the joint point spectrum.
In fact, if we consider a commutative algebra $L$, and  $T=(T_1,\ldots ,T_n)$ is an $n$-tuple of matrices such that $\{T_i\}_{(1\le i\le n)} $
is a basis of $L$, then the set of weights of $L$ for $E$ rapresented in terms of the basis
of $L^*$ dual to $\{T_i\}_{(1\le i\le n)}$ coincides with $\sigma_{pt}(T)$.\par

\indent As we shall work with the coordinates of elements of
$Sp(L,E)$, we need to construct suitable bases for $L$ and $L^*$. \par

\indent According to [5, Chapter IV, Section 1, Proposition 1], there is a Jordan-H\"older sequence of ideals, $(L_i)_{0\le i\le n}$,
such that:\par
\noindent (i) $\{0\}= L_{0}\subseteq L_i\subseteq L_n= L$,\par
\noindent (ii) $\dim L_i= i$,\par
\noindent (iii) there is a $k$ ($0\le k\le n$) such that $L_k= L^2$,\par
\noindent (iv) if $i< j$, then $[L_i ,L_j]\subseteq L_{i-1}$.\par
\indent As a consequence, if $\{ x_j\}_{1\le j\le n}$ is a basis of $L$
such that that $\{ x_j\}_{1\le j\le i}$ is a basis
of $L_i$ for each $i$, then
$$
[x_j,x_i]=\sum_{h=1}^{i-1} c^h_{ij}x_h,
$$
where $1\le i<j\le n$. Such a basis will be called a Jordan-H\"older basis. \par

\indent Now, given a Jordan-H\"older basis and an $\alpha\in Sp(L,E)$, we may 
represent it by its coordinates in the dual basis of $L^*$, i.e., by $(\alpha(x_1),\ldots , \alpha(x_n))$. The set of all such $n$-tuples
will be denoted by $Sp((x_i)_{1\le i\le n},E)$, i.e., 
$$
Sp((x_i)_{1\le i\le n},E)=\{(\alpha(x_1),\ldots , \alpha(x_n))\colon\alpha\in Sp(L,E) \}.
$$

\indent We observe that $Sp((x_i)_{1\le i\le n},E)=\sigma_{pt}((x_i)_{1\le i\le n})$.\par

\indent Recall that when $L$ is a solvable Lie algebra, by [5, Chapter V, Section 1, Proposition 2] we
may also construct a sequence  $(L_i)_{0\le i\le n}$ of ideals with properties (i),
(ii) and (iii), and property (iv)': if $i<j$, $[L_i,L_j]\subseteq L_i$. Thus, if $\{x_j\}_{1\le j\le n}$ is a basis of $L$
such that $\{x_j\}_{1\le j\le i}$ is a basis of $L_i$ for each $i$,
then, 
$$[x_j,x_i]=\sum_{h=1}^ic^h_{ij}x_h,$$ 
where $1\le i< j\le n$. As in the 
nilpotent case, such a basis will be called a Jordan-H\"older basis of the solvable Lie algebra $L$.\par

\indent We now review the geometric and the algebraic spectral radius,
as  defined in [1]. \par

\indent Let  $E_p$, $1\le p\le\infty$, denote the 
space $\mathbb{C}^{d}$ provided with the usual $p$-norm,
\begin{align*}
&\parallel x\parallel_p=(\sum_{i=1}^d \mid x_i\mid^p)^{1/p},\hskip2cm 1\le p<\infty,\\
&\parallel x\parallel_{\infty} =\sup_{1\le i\le d}\mid x_i\mid,\\
\end{align*}
where $x\in \mathbb{C}^d$.\par

On the other hand, for $\lambda\in \mathbb{C}^n$, we denote its $p$-norm by $\mid \lambda\mid_p$,
$1\le p\le \infty$.\par

\indent If $T=(T_1,\ldots ,T_n)$ is an $n$-tuple of $d\times d$-complex matrices,
then the geometric spectral radius of $T$ is defined as
$$
r_p(T)=max\{\mid\lambda\mid_p\colon \lambda\in \sigma_{pt}(T)\}.
$$

On the other hand, we may identify a matrix $M$ with the associated linear transformation, 
which we also denote by $M$. Thus, the $n$-tuple $T$ induces an operator from $E_p$ to the direct
sum of $n$ copies of $E_p$, considered with its natural $p$-norm. The norm of this operator,
also denoted by $T$, is 
$$
\parallel T\parallel_p=\sup_{\parallel x\parallel_p\le 1}(\sum_{j=1}^n \parallel T_j(x)\parallel_p^p)^{1/p}.
$$

\indent Now, given an $m\in \mathbb{N}$, we consider the $n^m$-tuple whose entries are
the products $T_{i_1}\ldots T_{i_m}$, where $1\le i_1,\ldots ,i_m\le n$, with repetitions
allowed and the indices arranged lexicographically. We denote this $n^m$-tuple 
by $T^m$. Them, the algebraic spectral radius of the $n$-tuple $T$ is defined as
$$
\varrho_p (T)=\hbox{\rm inf}\parallel T^m\parallel_p^{1/m},\hskip2cm 1\le p\le\infty.
$$

\indent As already mentioned, R. Bhatia and  T. Bhattacharyya proved in [1] that for
a commuting $n$-tuple $T$ of matrices  the algebraic and the geometrical spectral radius coincide. Now, given a nilpotent Lie algebra $L $ of matrices in ${\mathcal L}(\mathbb{C}^{d})^{op}$,
and  a Jordan-H\"older basis $\{ x_i \}_{1\le i\le n}$ of $L$, where $n$ is $\dim L$,
we may consider the $n$-tuple $(x_i)_{1\le i\le n}$. In the next section we
shall see that the geometric and the algebraic spectral radius of $(x_i)_{1\le i\le n}$ also coincide. \par

\section{The Main Result}

 \indent In this section we prove that the spectral radius formula proved in [1] 
extendes to nilpotent Lie algebras of matrices. \par

\indent Observe that if $T=(T_1,\ldots ,T_n)$ is a commuting $n$-tuple of matrices, 
and $U$ is an invertible matrix such that $UTU^{-1}=(UT_1U^{-1},\ldots ,UT_n U^{-1})$
is an $n$-tuple of commuting upper-triangular matrices, then $\sigma_{pt}(T)=\sigma_{pt}(UTU^{-1})$,
$r_p(T)=r_p(UTU^{-1})$ and $\varrho_p(T)=\varrho_p(UTU^{-1})$; see [1] of Proposition 2
below. If we decompose the matrices $UT_i U^{-1}$, $1\le i\le n$, in its diagonal and nilpotent part,
we have two new commuting $n$-tuples of matrices, $D$ and $N$, respectively. The spectral radii of $T$ coincide with the corresponding
radii of $D$. This suggests that in the computations 
of the spectral radii the nilpotent parts of the matrices are not of great importance. 
In this section, in order to prove our main theorem, we consider an $n$-tuple
of upper-triangular matrices with some additional properties, and we give an upper bound for the associated $n$-tuple $N$ in order to show that
the algebraic and the geometrical spectral radii of $T$ coincide. As the nilpotent and the commuting case
may be reduced to this situation, we obtain our main result as well as a new
proof for commuting case. \par

\indent Let us begin with two propositions which smplify our work.\par

\begin{pro} Let $T$ be an $n$-tuple of $d\times d$-matrices. Then 
$$
r_{\infty}(T)=\lim_{p\to \infty} r_p(T),\hskip2cm\varrho_{\infty}(T)=\lim_{p\to \infty}\varrho_p(T).
$$ 
\end{pro}
\begin{proof}

\indent If $x\in\mathbb{C}^q$, $q\in \mathbb{N}$, then,
$$
\parallel x\parallel_{\infty}\le\parallel x\parallel_p\le q^{1/p}\parallel x\parallel_{\infty}.
$$
In particular, by the definition of $r_p(T)$, 

$$
r_{\infty}(T)\le r_p(T)\le n^{1/p} r_{\infty}(T),
$$
which gives us the first part of the proposition.\par
\indent On the other hand, an easy calculation using  the above inequality
yields
$$
\parallel T\parallel_p\le (dn)^{1/p}\parallel T\parallel_{\infty}, \hskip2cm\parallel T\parallel_{\infty}\le \parallel T\parallel_p d ^{1/p}.
$$
Then, if $m\in \mathbb{N}$, we obtain 
$$
\parallel T^m\parallel_{\infty}\le\parallel T^m\parallel_p d^{1/p}\le (d^2n^m)^{1/p}\parallel T^m\parallel_{\infty},
$$
which implies that,
$$
\parallel T^m\parallel_{\infty}^{1/m}\le\parallel T^m\parallel_p^{1/m}d^{1/pm}\le d^{2/pm}n^{1/p}\parallel T^m\parallel^{1/m}_{\infty}.
$$

Thus,
$$
\varrho_{\infty}(T)\le \varrho_p(T)\le n^{1/p}\parallel T\parallel_{\infty},
$$
which shows that
$$
\lim_{p\to\infty}\varrho_p(T)=\varrho_{\infty}(T).
$$
\end{proof}
\indent Thus we may restrict our proof to the case $1\le p<\infty$.

\begin{pro} Let $T$ be an $n$-tuple of matrices in $E_p$, and $U$
an invertible matrix in $E_p$. Set $UTU^{-1}=(UT_1U^{-1},\ldots , UT_nU^{-1})$,
Then
$$
r_p(T)=r_p(UTU^{-1}), \hskip2cm \varrho_p(T)=\varrho_p(UTU^{-1}).
$$
\end{pro}
\begin{proof}

\indent As $\sigma_{pt}(UTU^{-1})=\sigma_{pt}(T)$, we have $r_p(UTU^{-1})=r_p(T)$.\par

\indent For  the algebraic spectral radius, we first observe that  $(UTU^{-1})^m=U(T^m)U^{-1}$ for all
$m\in\Bbb N$ . Thus, if $k$ is such that
$U^{-1}(B[0,k])\subseteq B[0,1]$, then, $\parallel (UTU^{-1})^m\parallel_p\le \parallel U\parallel_p\parallel T^m\parallel_p k^{-1}$,
which gives 
$$
\varrho_p(UTU^{-1})\le\varrho_p(T).
$$
However, as $T= U^{-1}(UTU^{-1})U$, we obtain the desired equality.\par
\end{proof}

\indent From now on we consider an $n$-tuple $T$ of matrices  such that $\mathbb{C}^d$ may
be decomposed into a direct sum of linear subspaces,  $\mathbb{C}^d=\oplus_{1\le i\le s}M_s$,
such that for each $i$, $1\le i\le n$, the linear transformation associated
with $T_i$ satisfies  $T_i(M_j)\subseteq M_j$, for $1\le j\le s$.
In addition, we assume that for each $j$ there is a basis of $M_j$ in which the matrix of $T_i\mid M_j $ has an upper triangular form for all $i$.
Moreover, we also assume that all the diagonal entries of $T_i\mid M_j$ coincide. 
By Proposotion 2 we may suppose that the above basis is the canonical one and 
that each $M_i$ is generated by elements of the canonical basis of $\mathbb{C}^d$.\par

\indent A straightforward calculation shows that for such $T$ we have  
$$
\sigma_{pt}(T)=\{(c^1_j,\ldots ,c^n_j)\colon 1\le j\le s\},$$
\noindent where $c_j^{i}$ denotes the diagonal entries of $T_i\mid M_j$ in the above basis of $M_j$.\par

\indent On the other hand, by  the theory of weight
spaces reviewed in Section 2, if $\{ x_i\}_{1\le i\le n}$ is a Jordan-H\"older basis of an $n$-dimensional nilpotent Lie algebra of linear
transformations defined on a complex finite dimensional vector space, then the $n$-tuple
$(x_i)_{1\le i\le n}$ clearly satisfy the above conditions.\par

\indent Moreover, if $T=(T_1,\ldots ,T_n)$ is a commuting $n$-tuple of $d\times d$-complex
matrices, and if $\tilde L$ is the $\mathbb{C}$-vector subspace of $\mathcal{L} (\mathbb{C}^d)$
generated by $T_i$, $1\le i\le n$, then $\tilde L$
is a commuting Lie subalgebra of $\mathcal{L} (\mathbb{C}^d)$ and $\mathcal {L} (\Bbb C^d)^{op}$,
in particular, $\tilde L$ is a nilpotent Lie subalgebra of $\mathcal {L}(\mathbb{C}^d)^{op}$.
Thus, we may apply the weight space theory to $\tilde L$ and $\mathbb{C}^d$ to
obtain a subspace decomposition  and a basis of $\mathbb{C}^d$ in which
the above conditions are satisfied by $T$. Indeed, as the properties (i)-(iv) of
the theory of weight spaces reviewed in Section 2 are satisfied by each $x$ in $\tilde L$, 
they are, in particular, satisfied by $T_i$, $1\le i\le n$. This approach
gives a
more precise description of the joint spectrum of an $n$-tuple of commuting matrices,
refining those  [1], [7] and [9]. \par

\indent Now denote by $D=(D_1,\ldots ,D_n)$, respectively $N=(N_1,\ldots ,N_n)$, the $n$-tuple 
of the diagonal, respectively nilpotent, parts of the matrices $T_i$, $1\le i\le n$. As $\sigma_{pt}(D)
=\sigma_{pt}(T)$, we have $r_p(D)=r_p(T)$. Furthermore, as $D$ is an $n$-tuple of commuting
matrices, by [1] Lemma 6, $\varrho_p (D)=r_p(D)$.\par

\indent If $s=1$ and if $D_i=c^{i}Id$, $1\le i\le n$, then
an easy calculation gives
$$
r_p(D)^m=\parallel (c^1,\ldots ,c^n)\parallel_p^m=(\sum_{(i_1,\ldots ,i_m)\in I_m}\prod^m_{k=1}\mid c^{i_k}\mid^p)^{1/p}=\parallel D^m\parallel_p,
$$
where $I_m=\{(i_1,\ldots ,i_m)\colon 1\le i_k\le n \hbox{ for } 1\le k\le m\}$.\par
\indent We now start the proof of our main result: for an $n$-tuple $T$
which satisfy the above conditions, $r_p(T)=\varrho_p (T)$.\par

\begin{pro} Let $T$, $D$ and $N$ be as above. Then,  
$$
r_p(T)\le \varrho_p(T). 
$$
\end{pro}
\begin{proof}

\indent Suppose that $D_i\mid M_j=c^{i}_j Id_j$ for all $1\le i\le n$, and $j$, $1\le j\le s$, where
 $Id_j$ denotes the identity of $M_j$. Then, as $r_p(T)=r_p(D)$, and as $D_i$ is the diagonal part of $T_i$,  if $(c_{j_0}^1,\ldots ,c_{j_0}^n)$ is such that 
$\parallel (c_{j_0}^1,\ldots ,c_{j_0}^n)\parallel_p =r_p(D)$, there is an element $x\in M_{j_0}$ such that $\parallel x\parallel_p=1$
and, for all $(i_1,\ldots ,i_m)\in I_m$,
$$
T_{i_1}\ldots T_{i_m}(x)= D_{i_1}\ldots D_{i_m}(x)=\prod_{1\le k\le m}  c_{j_0}^{i_k}x.
$$
Thus, by the previous observation,

\begin{align*}
r_p(T)^m &=r_p(D)^m=\parallel (c_{j_0}^1,\ldots ,c_{j_0}^n)\parallel_p^m =(\sum_{(i_1,\ldots ,i_m)\in I_m}\prod_{k=1}^m \mid c_{j_0}^{i_k}\mid^p)^{1/p}\\
         &=(\sum_{(i_1,\ldots , i_m)\in I_m}\parallel T_{i_1}\ldots T_{i_m}(x)\parallel_p^p)^{1/p}\le \parallel T^m\parallel_p,\\
\end{align*}
which implies that,
$$
r_p(T)\le \varrho_p(T).
$$
\end{proof}
\indent In order to prove that $\varrho_p(T)\le r_p(T) $ we need some prepartion.
We begin by studing the form of the a product of $m$ upper triangular matrices
with constant diagonal entries.\par

\indent Let $T=(T_1,\ldots ,T_n)$ be $n$-tuple of $b\times b$ upper triangular matrices
whose diagonal entries coincide, i.e., for each $i$ ($1\le i\le n$) there is a $c^i\in \mathbb{C}$
such that $(T_i)_{ii}=c^i$ for all $1\le t\le b$. Let $m\in \mathbb{N}$ and  $1\le i_k\le n$  for $1\le k\le m$. Then
$(T_{i_1}\ldots T_{i_m})_{st}=0$
if $1\le t<s\le b$, and $(T_{i_1}\ldots T_{i_m})_{st}=\prod_{k=1}^m c^{i_k}$ if $s=t$, .     
As $T_i$ are upper triangular, a straightforward calculation shows that
if $1\le s<t\le b$ then,
$$
(T_{i_1}\ldots T_{i_m})_{st}=\sum_{(h_0,\ldots , h_m)\in J}\prod_{1\le k\le m}(T_{i_k})_{h_{k-1}h_k},
$$
 
\noindent where $J=\{ (h_0,\ldots ,h_m)\colon s=h_0\le \ldots\le h_m=t\}$.
Decompose $J$ as $\cup_{0\le q\le m-1} J_q$, where $J_q=\{ (h_0,\ldots ,h_m)\in J\colon \hbox{ there are q-indices $k$ with } h_k=h_{k+1}, \hbox{} 0\le k\le m-1\}$.  
As $t>s$, there can not be $m$ such $k$. Moreover, if $J_q\ne \emptyset$ there are $m-q$ pairs
 $(k, k+1)$, $0\le k\le m-1$, such that $h_k\ne h_{k+1}$, which
implies that $t-s\ge m-q$. Thus, $J=\cup_{m-t+s\le q\le m-1} J_q$. \par

\indent For if $1\le q\le t-s$, we may represent $J_{m-q}$ as, 

\begin{align*}
J_{m-q}=\{ & (s,\ldots ,s,s+l_1,\ldots ,s+l_1,s+l_2, \ldots ,s+l_2,\ldots ,s+l_{q-1},\ldots ,s+l_{q-1},t,\ldots t)\colon\\
           & 1\le l_1 < \ldots < l_{q-1}\le t-s-1\},\\\end{align*}

\noindent where the jumps occur at the index $k_u$, $1\le u\le q$, and  $1\le k_1<\ldots < k_q\le m-1$. With this representation
is easy to see that $J_{m-q}$ has $\binom{m-1}{q}\binom{t-s-1}{q-1}$ elements, and that
$$
(T_{i_1}\ldots T_{i_m})_{st}=\sum_{1\le q\le t-s}\sum_{K_q,L_q}\prod_{k=1, k\ne k_u}^m  c^{i_k}\prod_{u=1}^q (T_{i_{k_u}})_{s+l_{u-1} s+l_u},
$$
\noindent where $K_q=\{(k_1,\ldots ,k_q)\colon 1\le k_1< \ldots <k_q\le m-1\}$ and $L_q=\{(l_1,\ldots ,l_{q-1})\colon 1\le l_1<\ldots < l_{q-1}\le t-s-1\}$, $1\le q\le t-s$.\par 

\indent Now we prove the reverse inequality of Proposition 3.3 for the case under
consideration.\par 
\indent If $\parallel y \parallel_p\le 1$, then
\begin{align*}
\parallel T_{i_1}\ldots T_{i_m}(y)\parallel_p^p &=\parallel (\sum_{v=1}^{b} (T_{i_1}\ldots T_{i_m})_{wv}y_v)_{1\le w\le b}\parallel_p^p\\
                                                & \le (\sum_{w=1}^{b} (\sum_{v=w}^{b} \mid (T_{i_1}\ldots  T_{i_m})_{wv}\mid)^p)\parallel y\parallel_p^p.\\\end{align*}

\noindent If $w=v$, we have that $\mid (T_{i_1}\ldots  T_{i_m})_{wv}\mid=\prod_{k=1}^m\mid c^{i_k} \mid$.
On the other hand, if $1\le w<v\le b$, as there is constant $R_1$, $R_1>1$, such that
$\mid (T_i)_{st}\mid \le R_1$ for all $s$, $t$ and $i$ with  $1\le s, t\le b$ and $1\le i\le n$, we have 
$$
\mid (T_{i_1}\ldots  T_{i_m})_{wv}\mid\le R_1^{b-1} (b-1)!\sum_{1\le q\le v-w}\sum_{(k_1,\ldots ,k_q)\in K_q} \prod_{k=1, k\ne k_u}^m \mid c^{i_k}\mid,
$$
\noindent Now, if $w_0$ and $v_0$ are such that  $\mid T_{1_1}\ldots T_{i_m}\mid_{wv}\le\mid T_{i_1}\ldots T_{i_m}\mid_{w_0v_0}$
for all $w$, $v$ with $1\le w<v\le b$,
we obtain 

$$
\parallel T_{i_1}\ldots  T_{i_m}(y)\parallel_p^p\le C\parallel y\parallel_p^p(\prod_{k=1}^m\mid c^{i_k}\mid^p + (\sum_{1\le q\le v_0-w_0}\sum_{(k_1\ldots ,k_q)\in K_q } \prod_{k=1, k\ne k_u}^m \mid c^{i_k}\mid)^p),
$$
where $C=R_1^{p(b-1)} bb!^p$.\par
\indent  Observe that if there are $l$ indices, $i_1,\ldots  ,i_l$ such that
$c^{i_j}=0$, and if $l> v_0-w_0$, then $T_{i_1}\ldots T_{i_m}=0$.\par
\indent Set $\overline{I_m}=\cup_{0\le r\le \tilde b }I^r_m$,
where $\tilde{b} =\hbox{\rm }min\{ b-1, m\}$ and  
$$
I^r_m=\{(i_1,\ldots ,i_m)\in I_m\colon \hbox{ there are $r$-indices $l$ such that
$c^{i_l}=0$} \}.
$$
Then, if $(i_1,\ldots ,i_m)\in I_m^0$, we have 

$$
\parallel T_{i_1} \ldots T_{i_m}\parallel _p^p\le C\parallel y\parallel_p^p\prod^m_{k=1}\mid c^{i_k}\mid^p(1+(\sum_{1\le q\le v_0-w_0}\sum_{(k_1\ldots ,k_q)\in K_q}\prod_{u=1}^q\frac{1}{\mid c^{i_{k_u}}\mid})^p).
$$
\noindent If $R_2\ge \hbox{\rm max}\{R_1,\frac{1}{\mid c^i\mid}\colon c^i\ne 0, 1\le i\le n\}$,
then

$$ 
\parallel T_{i_1}\ldots  T_{i_m}(y)\parallel_p^p\le C_0\parallel y\parallel_p^p\prod_{k=1}^m\mid c^{i_k}\mid^p ,
$$
\noindent where $C_0=2^pC(b-1)^p(m-1)^{(b-1)p}R_2^{p^2}$.

\indent Moreover, if $(i_1,\ldots ,i_m)\in I^r_m$ ($1\le r\le \tilde{b}$), a similar
calculation gives 
 
$$
\parallel T_{i_1}\ldots  T_{i_m}(y)\parallel_p^p\le C_r\parallel y\parallel_p^p\prod_{k=1, k\ne l}^m\mid c^{i_k}\mid^p, 
$$
\noindent where $i_l$ ($1\le l\le r$), are such that $c^{i_l}=0$, and $C_r=2^pC(b-r)^p(m-1)^{(b-1)p}R_2^{p^2}\le C_0$, 
for all $1\le r\le \tilde{b}$.\par
\indent Now,   $m\ge b$, an easy calculation shows that

\begin{align*}
\parallel T^m\parallel_p^p&=\sup_{\parallel y\parallel_p\le 1} (\sum_{(i_1,\ldots ,i_m)\in I_m}\parallel T_{i_1}\ldots T_{i_m}(y)\parallel_p^p)\\
                          &=\sup_{\parallel y\parallel_p\le 1} (\sum_{0\le r\le b-1}\sum_{(i_1,\ldots ,i_m)\in I_m^r}\parallel T_{i_1}\ldots T_{i_m}(y)\parallel_p^p)\\
                          &\le C_0\sum_{0\le r\le b-1}\sum_{(i_1,\ldots ,i_m)\in I_m^r} \prod^m_{k=1, c^{i_k}\ne 0}\mid c^{i_k}\mid^p\\
                          &=C_0\sum_{0\le r\le b-1} h^r\prod_{j=0}^{r-1}(m-j)\sum_{(i_1,\ldots ,i_{m-r})\in I_{m-r}} \prod^m_{k=1}\mid c^{i_k}\mid^p,\\\end{align*}

\noindent where $h=\sharp\{ l\colon c^l=0, 1\le l\le n\}$.\par
\indent By the observation which precedes Proposition 3.3,
$$
\parallel T^m\parallel_p^p\le C_0h^{b-1}m^br_p(D)^{p(m-b+1)}(\sum_{j=0}^{b-1}r_p(D)^{pj}).
$$
Thus, for $m\ge b$ there is a constant $K>1$ such that
$$
\parallel T^m\parallel_p^{1/m}\le K^{1/m}(m^{1/m})^{2b}r_p(D)^{(m-b+1)/m},
$$
which implies that
$$
\varrho_p(T)\le r_p(D)= r_p(T),
$$
and by Proposition 3.3 we obtain 
$$
\varrho_p(T)=r_p(D)=r_p(T).
$$
\indent Let us now return to the general case. For an $n$-tuple $T$  as specified after Proposition 3.2, and for $s=1$, we have
just seen that $\varrho_p(T)=r_p(T)$. We now prove the general case. Let us begin
by the following observation. By Proposition 3.2, as we may assume that the subspaces $M_j$ ($1\le j\le s$),
are generated by elements of the canonical basis of $\mathbb{C}^d$, if  $y=(y_j)_{1\le j\le s}\in \mathbb{C}^d$ and 
$\parallel y\parallel_p\le 1$, then 
                           
$$
\parallel T_{i_1}\ldots T_{i_m}(y)\parallel_p^p=\parallel \sum_{1\le j\le s} T_{i_1}\ldots T_{i_m}(y_j)\parallel_p^p=\sum_{1\le j\le s}\parallel T_{i_1}\ldots T_{i_m}(y_j)\parallel_p^p.
$$

\indent Now, as $\parallel y_j\parallel_p\le \parallel y\parallel_p\le 1$,
for $1\le j\le s$, we may proceed as follows.\par

\begin{align*}
\parallel T^m\parallel_p   &=\sup_{\parallel y\parallel_p\le 1}(\sum_{(i_1,\ldots ,i_m)\in I_m} \parallel T_{i_1}\ldots T_{i_j}(y)\parallel_p^p)^{1/p}\\  
                           &\le \sup_{\parallel y\parallel_p\le 1} (\sum_{(i_1,\ldots ,i_m)\in I_m} \sum_{1\le j\le s}\parallel T_{i_1}\ldots T_{i_m}(y_j)\parallel_p^p)^{1/p}\\     
                           &\le\sum_{1\le j\le s}\parallel (T^j)^m\parallel_p\\
                           &\le s\parallel (T^{j_0})^m\parallel_p,\\\end{align*}

\noindent where $T^j$ is the $n$-tuple defined by $T_i^j=T_i\mid M_j$ ($1\le i\le n$), and $j_0$ is such that
$\parallel(T^j)^m\parallel_p\le\parallel (T^{j_0})^m\parallel_p$ for all $j$.\par
\indent Thus, if $d_{j_0}=\dim M_{j_0}$, then, there is a constant $K_0$ such that,

\begin{align*}
\parallel T^m\parallel_p^{1/m}&\le s^{1/m}K_0^{1/m}{(m^{1/m})}^{2d_{j_0}}r_p(D^{j_0})^{(m-d_j+1)/m}\\
                           &\le s^{1/m}K_0^{1/m}{(m^{1/m})}^{2d_{j_0}}r_p(D)^{(m-d_j+1)/m},\\           \end{align*}    

\noindent which implies that,
$$
\varrho_p(T)\le r_p(D)=r_p(T),
$$
  
Thus,
$$
\varrho_p(T)=r_p(T).
$$

\indent  We have thus proved our main theorem.\par

\begin{thm}Let $L$ be a complex nilpotent Lie algebra of linear transformations in ${\mathcal L}(\mathbb{C}^d)^{op}$,
and $\{x_i\}_{1\le i\le n}$  a Jordan-H\"older  basis of $L$. Then,
$$
\varrho_p ((x_i)_{1\le i\le n})=r_p((x_i)_{1\le i\le n}).
$$
Moreover, if $T=(T_1,\ldots ,T_n)$ is an $n$-tuple of $d\times d$-commuting matrices,
then,
$$
\varrho_p(T)=r_p(T),
$$
where $1\le p\le \infty$.\end{thm}
\begin{proof}
\indent It is a consequence of Propositions 3.1-3.3 and the above calculations.\par
\end{proof}
\indent Finally, we consider several examples  to show that our result fails
in the solvable non nilpotent case.  \par
\indent  Consider a complex solvable finite dimensional Lie algebra of
linear transformations, $L$, acting on a complex finite dimensional vector space $E$. By [5, Chapter V, Section 1, Proposition 2],  as
in the nilpotent case, we may construct a Jordan-H\"older  sequence of ideals and a
Jordan-H\"older basis for $L$ (see Section 2). However, in the solvable non nilpotent case, Theorem A is
no longer true, i.e., if $\dim L=n$ and $\{x_i\}_{1\le i\le n}$ is a Jordan-H\"older basis
of $L$, $Sp((x_i)_{1\le i\le n},E)\ne \sigma_{pt}((x_i)_{1\le i\le n})$.
Thus, we may consider $\hbox{\rm max}\{\mid\lambda\mid_p\colon \lambda\in Sp((x_i)_{1\le i\le n},E)\}$
instead of $r_p((x_i)_{1\le i\le n})$, and  try to see if $\varrho_p((x_i)_{1\le i\le n})=\hbox{\rm max}\{\mid\lambda\mid_p:
\lambda\in Sp((x_i)_{1\le i\le n}, E)\}$. However, as we shall see, this  equality
also fails.\par
\indent  In the following examples we consider the space $E=\mathbb{C}^2$, and the solvable non nilpotent algebras
we work with are representations in ${\mathcal L}(\mathbb{C}^2)^{op}$ of the two dimensional
solvable algebra $L=<y>\oplus<x>$, with  bracket $[x,y]=y$. Obseerve that  
$(y,x)$ is a Jordan-H\"older basis of $L$. \par
\indent As our first example consider the algebra $L_1\subseteq {\mathcal L}(\mathbb{C}^2)^{op}$,
with Jordan-H\"older  basis,
$$
y=\begin{pmatrix} 0&2\\
           0&0\\\end{pmatrix} ,    \hskip2cm x=\begin{pmatrix} -1/2&0\\
                                                         0&1/2\\\end{pmatrix} . 
$$

Then
$$
\sigma_{pt}((y,x))=\{(0,-1/2)\}, \hskip2cm Sp((y,x),E)=\{(0,1/2), (0,-3/2)\},
$$
and 
$$
r_{\infty}((y,x))=1/2=\varrho_{\infty}((y,x)),\hbox{ but } \hbox{\rm max}\{\mid\lambda\mid_{\infty}\colon \lambda\in Sp((y,x),E)\}=3/2.
$$

\indent On the other hand, if $L_2\subseteq{\mathcal L}(\mathbb{C}^2)^{op}$ is the algebra
with Jordan-H\"older basis,
$$
y=\begin{pmatrix} 0&1\\
           0&0\\\end{pmatrix} ,\hskip2cm x=\begin{pmatrix} 2&0\\
                                                  0&3\\\end{pmatrix},
$$
we obtain 

$$
\sigma_{pt}((y,x))=\{(0,2)\}, \hskip2cm Sp((y,x),E)=\{(0,1), (0,3)\},
$$
and 
$$
r_{\infty}((y,x))=2,\hbox{ but }\varrho_{\infty}((y,x))=3=\hbox{\rm max}\{\mid\lambda\mid_{\infty}\colon\lambda\in Sp((y,x),E)\}.
$$ 
\indent Our last example is the algebra $L_3\subseteq{\mathcal L}(\mathbb{C}^2 )^{op}$  generated
by,
$$
y=\begin{pmatrix} 0&1\\
           0&0\\\end{pmatrix} ,    \hskip2cm x=\begin{pmatrix} -1/3&0\\
                                                         0&2/3\\\end{pmatrix} . 
$$

Then, 
$$
\sigma_{pt}((y,x))=\{(0,-1/3)\}, \hskip2cm Sp((y,x),E)=\{(0,-4/3), (0,2/3)\},
$$
and 
$$
r_{\infty}((y,x))=1/3,\hskip1cm \varrho_{\infty}((y,x))=2/3,\hskip1cm \hbox{\rm max}\{\mid\lambda\mid_{\infty}\colon\lambda\in Sp((y,x),E)\}=4/3.
$$

\bibliographystyle{amsplain}

\begin{thebibliography}{99}
\bibitem{ }R. Bhatia and T. Bhattacharyya, On the joint spectral radius of commuting matrices,
Studia Math. 114 (1995), 29-38.

\bibitem{ }E. Boasso, Dual properties and joint spectra for solvable
Lie algebras of operators, J. Operator Theory 33 (1995), 105-116.

\bibitem{ } E. Boasso, Joint spectra and nilpotent Lie algebras of Linear transformations, Linear Algebra Appl. 263 (1997), 49-62.

\bibitem{ } E. Boasso and A. Larotonda, A spectral theory for solvable
Lie algebras of operators, Pacific J. Math. 158  (1993), 15-22.

\bibitem{ }N. Bourbaki,  \'El\'ements de Math\'ematique, Groupes et Alg\`ebres de Lie,
Alg\`ebres de Lie Fasc. XXVI, 1960.

\bibitem{ } M. Ch\={o} and T. Huruya, On the joint spectral radius, Proc. Roy. Irish Acad.
Sect. A 91(1991), 39-44.

\bibitem{ } M. Ch\={o} and M. Takaguchi, Identity of Taylor's joint spectrum and Dash's joint spectrum, Studia Math. 70 (1982), 225-229.

\bibitem{ } N. Jacobson, Lie Algebras, Interscience Publishers, 1962.

\bibitem{ } A. McIntosh, A. Pryde and W. Ricker, Comparison of joint spectra
for certain classes of commuting opertors, Studia Math. 88 (1988),  23-36.

\bibitem{ } C. Ott, A note on a paper of E. Boasso and A. Larotonda, 
Pacific  J. Math. 173 (1996), 173-179.
 
\bibitem{ } Z. S\l odkowski, An infinite family of joint spectra, Studia
Math. 61 (1973), 239-255.

\bibitem{ } J.L.Taylor, A joint spectrum for several commuting operators,
J. Funct. Anal. 6 (1970), 172-191.


\end{thebibliography}

\vskip.5cm
\noindent Enrico Boasso\par
\noindent E-mail address: enrico\_odisseo@yahoo.it
\end{document}